\documentclass[12pt,reqno,oneside]{amsart}
\usepackage{amsmath,amsthm,amsfonts,amssymb}
\usepackage[mathscr]{eucal}
\usepackage{bbm}
\usepackage{indentfirst}
\usepackage{url}

\theoremstyle{plain}
\newtheorem{teo}{Theorem}[section]

\newtheorem{lem}[teo]{Lemma}

\theoremstyle{definition}

\newtheorem{obs}[teo]{Observation}

\numberwithin{equation}{section}

\def\bb1{{\mathbbm{1}}}

\begin{document}
	\baselineskip=22pt
	\title[Coverage of complete graphs with the frog model]{Critical Conditions for the Coverage of Complete Graphs with the Frog Model}
	\author{Gustavo~O.~de Carvalho}
	\author{F\'abio~P.~Machado}
	\address[F\'abio~P.~Machado]
	{Institute of Mathematics and Statistics
		\\ University of S\~ao Paulo \\ Rua do Mat\~ao 1010, CEP
		05508-090, S\~ao Paulo, SP, Brazil - fmachado@ime.usp.br}
	\noindent
	\address[Gustavo~O.~de Carvalho]
	{Institute of Mathematics and Statistics
		\\ University of S\~ao Paulo \\ Rua do Mat\~ao 1010, CEP
		05508-090, S\~ao Paulo, SP, Brazil - gustavoodc@ime.usp.br}
	\noindent
	\thanks{Research supported by CAPES (88887.676435/2022-00), FAPESP (23/13453-5)}
	
	\date{\today}

\begin{abstract}
We consider a system of interacting random walks known as the frog model. Let $\mathcal{K}_n=(\mathcal{V}_n,\mathcal{E}_n)$ be the complete graph with $n$ vertices and $o\in\mathcal{V}_n$ be a special vertex called the root. Initially, $1+\eta_o$ active particles are placed at the root and $\eta_v$ inactive particles are placed at each other vertex $v\in\mathcal{V}_n\setminus\{o\}$, where $\{\eta_v\}_{v\in \mathcal{V}_n}$ are i.i.d. random variables.
At each instant of time, each active particle may die with probability $1-p$.
Every active particle performs a simple random walk on $\mathcal{K}_n$ until the
moment it dies, activating all inactive particles it hits along its path. Let $V_\infty(\mathcal{K}_n,p)$ be the total number of visited vertices by
some active particle up to the end of the process, after all active particles have died. In this paper, we show that $V_\infty(\mathcal{K}_n,p_n)\geq (1-\epsilon)n$ with high probability for any fixed $\epsilon>0$ whenever $p_n\rightarrow 1$. Furthermore, we establish the critical growth rate of $p_n$ so that all vertices are visited. Specifically, we show that if $p_n=1-\frac{\alpha}{\log n}$, then $V_\infty(\mathcal{K}_n,p_n)=n$ with high probability whenever $0<\alpha<E(\eta)$ and $V_\infty(\mathcal{K}_n,p_n)<n$ with high probability whenever $\alpha>E(\eta)$.
\end{abstract}

\maketitle

\section{Introduction and results}

We study the frog model, which can be described by the following rules. Initially, particles are placed at the vertices of a connected graph $\mathcal{G}$. At time zero, all particles are inactive, except for those at a special selected vertex from $\mathcal{G}$ called the root, which are active. Active particles perform independent, simple, nearest-neighbor, discrete time random walks on $\mathcal{G}$, activating any inactive particles found on the vertices visited by their walk. Inactive particles remain on their initial vertex until activated (if it happens). This model is often associated with the dynamics of the spread of a rumor or an infection in a population. A formal definition of the frog model can be found in \cite{phase_transition}.

Several variations of the frog model have been studied. Most of the results are based on cases where $\mathcal{G}$ is infinite, usually the hypercubic lattice $\mathbb{Z}^d$ or the regular tree $\mathbb{T}^d$. Part of the work has focused on finding conditions for recurrence and transience, i.e., conditions for the root to be visited infinitely often \cite{recorrencia,recorrencia2,recorrencia3,recorrencia_drift3,recorrencia_drift4,recorrencia_drift5}. Another area of interest is describing the set of visited vertices and its asymptotic shape \cite{shape_theorem,shape_theorem_drift,drift,shape_theorem3}. For particles having a finite random lifetime, specifically those with a geometric distribution, some papers focus on studying conditions for the survival of the frog model \cite{phase_transition,phase_transition2,phase_transition5}. In finite graphs, there are works that deal with cases where particles have fixed finite lifetime \cite{vida_fixa2,cover_time,vida_fixa}, as well as some other lifetime schemes \cite{grafo_completo_fabio,vertice_visitado,grafo_completo_elcio,coverage}.

We consider the geometric lifetime version of the frog model. At each instant of time, independently of everything else, each active particle survives with probability $p\in(0,1)$ or dies otherwise, being permanently removed from the model. This means that the lifetime of active particles follow a geometric distribution supported on $\{0,1,...\}$ with parameter $1-p$.

Let $\mathcal{K}_n$ be the $n$-complete graph, i.e., $\mathcal{K}_n=(\mathcal{V}_n,\mathcal{E}_n)$ where $\mathcal{V}_n$ is the set of vertices with $|\mathcal{V}_n|=n$ and $\mathcal{E}_n=\{\{x,y\}|x,y\in\mathcal{V}_n,x\neq y\}$ is the set of undirected edges linking every pair of distinct vertices. Let $o\in\mathcal{V}_n$ be a special vertex called the root of $\mathcal{K}_n$. Initially, we consider that each vertex $v\in\mathcal{V}_n\setminus\{o\}$ has $\eta_v$ inactive particles and $o$ has $1+\eta_o$ active particles, where $\{\eta_v\}_{v\in\mathcal{V}_n}$ are i.i.d. with the same distribution as a variable $\eta$. The extra particle at the root is simply a way to ensure that the frog model does not end right from the start, but all of our results also hold without this extra particle if we consider the conditional probability on the event $\{\eta_o\geq 1\}$.

Excluding trivial cases where $P(\eta=\infty)>0$ or $p=1$, there are finite particles which could be activated since $\mathcal{V}_n$ is finite and each particle has a finite lifetime (given by a geometric distribution). Thus, there is a finite instant in which the process is over, i.e., there are no more active particles alive. Let $V_\infty(\mathcal{K}_n,p)$ be the number of visited vertices when the frog model is over. When either of the trivial cases above happens, it could be the case that the process is not over at any instant of time and all vertices are visited; in this case we simply say that $V_\infty(\mathcal{K}_n,p):=n$.

When there is exactly one particle per vertex ($\eta \equiv 1$), \cite[Theorem 1.2]{coverage} shows that both $V_\infty(\mathcal{K}_n,p)\leq c \log n$ and $V_\infty(\mathcal{K}_n,p)\geq c'n$ have positive asymptotic probabilities for $p>1/2$ and suitable constants $c>0$ and $c'\in(0,1)$. The following theorem shows that when $p=p_n$ such that $p_n\rightarrow 1$, w.h.p.\footnote{We say that a sequence of events $(E_n)_{n\in \mathbb{N}}$ happens with high probability (w.h.p.) when $\lim_{n\to \infty} P(E_n)=1$.} only the case $V_\infty(\mathcal{K}_n,p)\geq c'n$ is possible. More than that, the same goes for every $c'\in(0,1)$, therefore w.h.p. any arbitrary large proportion of vertices is visited.

\begin{teo}\label{teo:alta_proporcao}
    Let $\eta$ be such that $P(\eta=0)<1$. If $(p_n)_{n\in \mathbb{N}}$ is a sequence such that $\lim_{n \to \infty}p_n=1$, then for any $\epsilon \in (0,1)$,
    \[\lim_{n\to \infty}P(V_\infty(\mathcal{K}_n,p_n)\geq (1-\epsilon)n)=1.\]
\end{teo}

Theorem \ref{teo:alta_proporcao} may appear to suggest that all vertices of $\mathcal{K}_n$ are visited when $p_n\rightarrow 1$, but this is not necessarily the case. Actually, there is a significant difference between visiting a large proportion of vertices and visiting all vertices, similar to what occurs in another process that will later be used in a coupling with the frog model: the coupon collector's problem. As the collection increases, it becomes harder and harder to find new coupons in the coupon collector's problem, thus obtaining $(1-\epsilon)n$ different coupons demand roughly $cn$ draws, where $c=c(\epsilon)$ is a constant; while obtaining all $n$ different coupons demand roughly $n\log n$ draws. 

Next theorem shows the rate at which $p_n$ must grow in order to have w.h.p. all vertices visited. For any $x\in\mathbb{R}$, let $x^+:=\max\{0,x\}$ be its positive part.

\begin{teo}\label{teo:transicao_fase}
    Let $\eta$ be such that $P(\eta=0)<1$. Consider that $p_n=(1-\frac{\alpha}{\log n})^+$.
    
\begin{itemize}
    \item (i) If $\alpha>E(\eta)$, then $\lim_{n \to \infty}P( V_\infty(\mathcal{K}_n,p_n)=n)=0$.
    \item (ii) If $0<\alpha<E(\eta)$, then $\lim_{n \to \infty}P(V_\infty(\mathcal{K}_n,p_n)=n)=1$.
\end{itemize}
\end{teo}

\begin{obs}
When $E(\eta)=+\infty$, Theorem \ref{teo:transicao_fase} part $(ii)$ applies for every $\alpha>0$ and we have that $\lim_{n \to \infty}P(V_\infty(\mathcal{K}_n,(1-\frac{\alpha}{\log n})^+)=n)=1$.
\end{obs}


Theorem $\ref{teo:transicao_fase}$ relates to \cite[Proposition 1.1]{vida_fixa2}, where the authors address the frog model on $\mathcal{K}_n$ with fixed finite lifetime. They study the  susceptibility of $\mathcal{K}_n$ (among other graphs), essentially determining the shortest particles’ lifespan required for the entire graph to be visited by active particles. In addition to having a different particle lifetime scheme, that paper focuses on cases where $\eta$ are Poisson distributed. The approach used there does not apply to our case.

Simple coupling techniques show that $P(V_\infty(\mathcal{K}_n,p)=n)$ is non-decreasing in $p$. So, Theorem \ref{teo:transicao_fase} is also useful for several sequences $(p_n)_{n\in\mathbb{N}}$, depending on their asymptotic behavior. We have that $\lim_{n\to\infty}P(V_\infty(\mathcal{K}_n,p)=n)=0$ whenever $p_n\leq 1-\frac{\alpha}{\log n}$ for some $\alpha>E(\eta)$ and sufficiently large $n$; we also have that $\lim_{n \to \infty}P(V_\infty(\mathcal{K}_n,p_n)=n)=1$ whenever $p_n\geq 1-\frac{\alpha}{\log n}$ for some $0<\alpha<E(\eta)$ and sufficiently large $n$.

It is also interesting to note what happens when $(p_n)_{n\in \mathbb{N}}$ is a sequence that slowly converges to 1. When $\lim_{n\to \infty}p_n=1$ and $p_n\leq 1-\frac{\alpha}{\log n}$ for some $\alpha>E(\eta)$ and sufficiently large $n$, then $(1-\epsilon)n\leq V_\infty(\mathcal{K}_n,p_n)<n$ w.h.p. for any $\epsilon\in(0,1)$ by Theorems $\ref{teo:alta_proporcao}$ and $\ref{teo:transicao_fase}$. In this case, there is a non-empty set of unvisited vertices at the end of the process, but its cardinality is smaller than any pre-established proportion.

The remainder of this article is organized as follows. In Section \ref{sec:aux}, we provide some useful lemmas, as well as definitions and results about models that can be coupled with the frog model on $\mathcal{K}_n$. The proofs of all theorems are presented in Section \ref{sec:demo_kn}.

\section{Auxiliary results}\label{sec:aux}

\subsection{The coupon collector's problem}

Consider the classical coupon collector's problem. There are $n$ coupons and a collector who acquires a random and uniformly selected coupon at each instant of time. Let $\tau_i$ be the number of draws needed to obtain $i\in\{1,...,n\}$ different coupons, and $\tau:=\tau_{n}$ be the number of draws needed to complete the entire collection. 

Note that $\tau_{j+1}-\tau_j\sim Geo_1(\frac{n-j}{n})$\footnote{We let $Geo_0(.)$ and $Geo_1(.)$ denote geometric distributions with support on $\{0,1,..\}$ and on $\{1,2,...\}$, respectively.}, $j\in \{0,1,...,n-1\}$, where we define $\tau_0:=0$. Then, $\tau_{j+1}-\tau_j$ has the same distribution as the number of steps that active particles take (if there are particles that have enough lifetime to do so) to go from $j+1$ to $j+2$ visited vertices in $\mathcal{K}_{n+1}$. Similarly, $\tau=\sum_{j=0}^{n-1} (\tau_{j+1}-\tau_j)$ is related to the potential coverage of $\mathcal{K}_{n+1}$ in the frog model.

Next, we show two useful lemmas regarding the coupon collector's problem. We only show the proof for the second lemma. The first one has a simple proof involving expectation and variance of $\tau$, which can be seen in \cite[Example 2.2.7]{ccp2}. 

\begin{lem} \label{lema:ccp_concentrado}
For any $\epsilon\in(0,1)$,
\[\lim_{n \to \infty}P[(1-\epsilon)n\log n\leq \tau\leq (1+\epsilon)n \log n]=1.
\]
\end{lem}

\begin{lem}\label{lema:ccp_parcial}
For any $\epsilon\in(0,1)$,
\[\lim_{n \to \infty} P(\tau_{\lceil (1-\epsilon)n\rceil}\leq 2(\frac{1-\epsilon}{\epsilon})n)=1.\]
\end{lem}

\begin{proof}

It is easy to see, from the cumulative distribution function, that $Q_1 \succeq Q_2$\footnote{For any two random variables $X$ and $Y$ in the same probability space, we say that $X$ \textit{stochastically dominates} $Y$ when $P(X\geq a)\geq P(Y\geq a)$ for all $a\in\mathbb{R}$. We denote this relation by $X\succeq Y$.}  when $Q_1\sim Geo_1(q_1)$ and $Q_2\sim Geo_1(q_2)$ with $q_1\leq q_2$. Combining this with the fact that $ \frac{n-i}{n}\geq \frac{n-(1-\epsilon)n}{n}=\epsilon$ for $i\in \{0,1,...,\lceil (1-\epsilon)n\rceil-1\}$, we have that
\begin{equation}\label{eq:ccp1}
\tau_{\lceil (1-\epsilon)n\rceil}=\sum_{i=0}^{\lceil (1-\epsilon)n\rceil-1} (\tau_{i+1}-\tau_i)\preceq \sum_{i=1}^{\lceil (1-\epsilon)n\rceil}Z_i,\end{equation}
where $Z_i\stackrel{i.i.d.}{\sim} Geo_1(\epsilon)$.

Furthermore, as $E(\sum_{i=1}^{\lceil (1-\epsilon)n\rceil}Z_i)=\frac{{\lceil (1-\epsilon)n\rceil}}{\epsilon}$ and $\frac{(1-\epsilon)n}{\lceil (1-\epsilon)n\rceil}\stackrel{n\to \infty}{\to} 1$, we have for sufficiently large $n$ that
\begin{equation}\label{eq:ccp2}
\frac{2(1-\epsilon)n}{\epsilon \lceil (1-\epsilon)n\rceil}-\frac{E(\sum_{i=1}^{\lceil (1-\epsilon)n\rceil}Z_i)}{ \lceil (1-\epsilon)n\rceil}\geq \frac{1}{2\epsilon}.
\end{equation}

By ($\ref{eq:ccp1}$) and ($\ref{eq:ccp2}$),
\[\begin{aligned}P(\tau_{\lceil (1-\epsilon)n\rceil}\geq \frac{2(1-\epsilon)}{\epsilon}n)&\leq P(\sum_{i=1}^{\lceil (1-\epsilon)n\rceil}Z_i\geq \frac{2(1-\epsilon)}{\epsilon}n)\\
&\leq P(\frac{\sum_{i=1}^{\lceil (1-\epsilon)n\rceil}Z_i-E(\sum_{i=1}^{\lceil (1-\epsilon)n\rceil}Z_i)}{\lceil (1-\epsilon)n\rceil}\geq \frac{1}{2\epsilon}).\end{aligned}\]

Therefore, $\lim_{n\to\infty}P(\tau_{\lceil (1-\epsilon)n\rceil}\geq \frac{2(1-\epsilon)}{\epsilon}n)=0$ by the law of large numbers.
\end{proof}

\subsection{Auxiliary process}

Here we establish an auxiliary process analogous to the one in \cite{coverage}. The auxiliary process can be thought of as the frog model on a connected graph $\mathcal{G}=(\mathcal{V},\mathcal{E})$, but in such a way that only one particle is allowed to move per round, and additional particles are included so that the process continues infinitely.

As in the frog model, a vertex $o\in\mathcal{V}$ is called the root of $\mathcal{G}$. At time zero, every vertex $v\in \mathcal{V}\setminus\{o\}$ has $\eta_v$ inactive particles while $o$ has $1+\eta_o$ active particles, but $\{\eta_v\}_{v\in \mathcal{V}}$ are not necessarily i.i.d.. Let all $1+\sum_{i\in\mathcal{V}}\eta_i$ initial particles be denoted as originals. This is done because \textit{extra} particles will be added to the process at some point. Given the initial configuration, the auxiliary process on $\mathcal{G}$ with survival parameter $p$ has the following rules:

\begin{itemize}
    \item At each instant $k$ of time, one active particle is chosen (the second and third rules will guarantee that there is at least one active particle) to participate in round $k$. This particle can be chosen randomly or by any selection rule. The chosen particle checks whether it survives with probability $p$ and, if so, moves to a randomly and uniformly chosen connected vertex, activating any inactive particles at that vertex.
    \item Let $R=R(\mathcal{G},p)$ be the moment at which the only living original active particle dies, that is, the instant at which the original process ends. At this same moment, every remaining inactive particle is now called extra inactive. Furthermore, a new extra active particle is injected at the root $o\in\mathcal{V}$ to continue the process.
    \item  After $R$, each time the only living active particle dies, a new extra active particle is added at the root $o\in\mathcal{V}$.
\end{itemize}

Note that there are only extra particles participating from round $R+1$ onwards. Excluding cases in which $\sum_{i\in \mathcal{V}}\eta_i=\infty$ or $p = 1$, we have that $R<\infty$ with probability 1, since at most $1+\sum_{i \in\mathcal{V}}\eta_i$ original particles could be activated, each with a finite lifetime given by a geometric distribution with parameter $1-p$. Despite $R<\infty$, the auxiliary process as a whole is infinite, as it always renews itself by injecting new extra particles on the root.

By this construction, it is possible to guarantee that there is always an active particle (original or extra) to participate in the $k$-th round for any $k\in \mathbb{N}$, regardless of whether the original process has ended or not. For $k\in\mathbb{N}$, define $X_k=X_k(\mathcal{G},p)$ depending on what happens to the particle that participates in round $k$:

\begin{itemize}
    \item $X_k=0$ if it dies.
    \item $X_k=1$ if it survives and goes to a vertex which has been visited before round $k$; or if it survives and goes to a vertex $v$ such that $\eta_v=0$ and that has never been visited before round $k$.
    \item $X_k=j$, $j\in\{2,3,...\}$, if it survives and goes to a vertex $v$ such that $\eta_v=j-1$ and that has never been visited before round $k$.
\end{itemize}

Note that the previous random variable can be interpreted as the number of descendants (in the sense of a branching process) of the particle participating in round $k$. From the number of descendants in each round, we define the number $A'_k=A'_k(\mathcal{G},p)$ of potentially active particles at the end of the $k$-th round as
\begin{equation} \label{def:pot_ativas}
A'_0:=1+\eta_o; \hspace{5mm} A'_k:=1+\eta_o+\sum_{j=1}^k (X_j-1), \hspace{1mm} k\in \{1,2,...\}.\end{equation}

The term \textit{potential} comes from the fact that the previous expression may also be counting the descendants of extra particles. We can disregard this by denoting $R=R(\mathcal{G},p)$ as
\begin{equation}\label{def:r}
R=\inf\{k:A'_k=0\}
\end{equation}
and defining the number $A_k=A_k(\mathcal{G},p)$ of active particles at the end of the $k$-th round as
\begin{equation} \label{def:ativas}
A_k:=A'_k\mathbbm{1}_{(k<R)}, \hspace{2mm} k\in \{0,1,2,...\}.
\end{equation}

Let $E_k\subset \{X_k\geq 1\}$ be the event in which the particle participating in round $k$ survives and visits a vertex that has never been visited before. Similarly to what was done previously, we define the number $V'_k=V'_k(\mathcal{G},p)$ of potentially visited vertices until the end of the $k$-th round as
\begin{equation}\label{def:pot_visitados}
    V'_0:=1; \hspace{5mm} V'_k:=1+\sum_{j=1} ^{k} \mathbbm{1}_{E_j}, \hspace{1mm} k\in \{1,2,...\} 
\end{equation}
and the number $V_k=V_k(\mathcal{G},p)$ of visited vertices until the end of the $k$-th round as
\begin{equation} \label{def:visitados}
V_k:=V'_k \mathbbm{1}_{(k<R)}+V'_R \mathbbm{1}_{(k \geq R)},\hspace{1mm} k\in \{0,1,2,...\}.
\end{equation}

We also define the total number $V_\infty=V_\infty(\mathcal{G},p)$ of visited vertices during the entire process by
\begin{equation} \label{def:visitados_total}
V_\infty:=\lim_{k \to \infty} V_k=V_{R}.
\end{equation}

Note that $V_\infty=V_R$ is not affected by the extra particles injected into the process, as they only play a role after $R$. The extra particles are only useful for the mathematical convenience they bring, such as allowing the evaluation of the number of potentially active particles and potentially visited vertices without needing to worry about whether all the original particles have already died or not. It is also important to note that allowing only one particle to move per round maintains the same final number of visited vertices as the original dynamics of the frog model. Therefore, the total number of visited vertices in the corresponding frog model is also described by $V_\infty$, making it possible to study the frog model through the auxiliary process.

Specifically when we consider complete graphs and that $\{\eta_v\}_{v\in\mathcal{V}}$ is a collection of i.i.d. random variables with the same distribution of $\eta$, we have that
\begin{equation}\label{eq:dist_x}
P(X_j(\mathcal{K}_{n+1},p)=x|V'_{j-1}(\mathcal{K}_{n+1},p)=v)=\begin{cases}
1-p \textit{, if $x=0$}\\
\frac{p (v-1)}{n}+p(\frac{n-v+1}{n})P(\eta=0) \textit{, if $x=1$}\\
p(\frac{n-v+1}{n})P(\eta=1) \textit{, if $x=2$}\\
\vdots\\
p(\frac{n-v+1}{n})P(\eta=k-1) \textit{, if $x=k$}\\\\
\vdots\\
\end{cases}
\end{equation}

\subsection{Helpful lemmas}

Here we provide two lemmas for later use. The first one relates the number of active particles to the total number of steps they will take before dying. This lemma will be particularly useful in combination with the results of the coupon collector's problem, as it gives us the necessary number of active particles such that the total remaining steps are greater than what is needed to visit a high proportion of vertices or even the entire graph.

\begin{lem}\label{lema:vida_total}
Consider the auxiliary process on any graph and with survival parameter $p_n$. Consider that at some instant $k=k(n)$ there are $A_k$ active particles and define $T(k)$ as the combined number of remaining steps that these $A_k$ particles will still take (i.e., $T(k)$ is the sum of the remaining lifetimes of these $A_k$ particles), with $T(k):=0$ if $A_k=0$. Also consider that $p_n\neq 1$ for all $n\in\mathbb{N}$ and $\lim_{n\to \infty}p_n=1$ and define $q_n:=(1-p_n)^{-1}$. For any $\epsilon>0$ and $b>0$,
\begin{itemize}
    \item (i) $\lim_{n \to \infty} P(T(k)\leq (1-\epsilon/2)bn q_n|A_k\leq (1-\epsilon)bn)=1;$
    \item (ii) $\lim_{n \to \infty} P(T(k)\geq (1+\epsilon/2)bn  q_n|A_k\geq (1+\epsilon)bn)=1.$
\end{itemize}
\end{lem}

\begin{proof}
    Let $T:=T(k)$ be the combined total number of steps that all particles that are active at time $k$ will still take. Let $\{T_i\}_{i\geq 1}$ be a sequence of random variables such that $T_i \stackrel{i.i.d.}{\sim} Geo_0(1-p_n)$. Note that $T$ has the same distribution as $\sum_{i=1}^{A_k}T_i$ because, due to lack of memory, we can ignore how many times each particle has survived before time $k$. Also note that
\begin{equation}\label{eq:t_<>}
    \begin{cases}T \preceq \sum_{i=1}^{(1-\epsilon)b n} T_i \text{ under the conditional of item } (i),\\
    T \succeq \sum_{i=1}^{(1+ \epsilon)b n} T_i\text{ under the conditional of item }(ii).
    \end{cases}
\end{equation}

As $T_i \sim Geo_0(1-p_n)$, we have that $E(T_i)=q_n -1$ and $Var(T_i)=q_n^2-q_n$. Then,
\begin{equation}\label{eq:mean_lemma}
E(\sum_{i=1}^{(1\pm\epsilon)b n} T_i)=(1\pm\epsilon)bn q_n-(1\pm\epsilon)bn\end{equation} 
and
\begin{equation}\label{eq:var_lemma}
Var(\sum_{i=1}^{(1\pm\epsilon)b n} T_i) =(1\pm\epsilon)bn(q_n^2-q_n).\end{equation}

Since $q_n$ diverges as $p_n\to 1$ by assumption, then $(1+\epsilon)bn\leq (\epsilon/4)bn q_n$ for sufficiently large $n$, which implies by ($\ref{eq:mean_lemma}$) that
\begin{equation}\begin{aligned}\label{eq:aux_lemma_+}
(1+\epsilon/2)bn q_n-E(\sum_{i=1}^{(1+\epsilon)bn} T_i)&=-(\epsilon/2)bn q_n +(1+\epsilon)bn\\
&\leq -(\epsilon/4) bn q_n.
\end{aligned}\end{equation}

On the other hand, it is also true that
\begin{equation}
\begin{aligned}\label{eq:aux_lemma_-}
(1-\epsilon/2)bn q_n-E(\sum_{i=1}^{(1-\epsilon)bn} T_i)&=(\epsilon/2)bn q_n +(1-\epsilon)bn\\
&\geq (\epsilon/4) bn q_n.
\end{aligned}
\end{equation}

By (\ref{eq:t_<>}), (\ref{eq:var_lemma}), (\ref{eq:aux_lemma_+}) and Chebyshev's inequality, we have for sufficiently large $n$ that
\begin{equation}\label{eq:t_lemma}
\begin{aligned}
P(T\leq (1+\epsilon/2)bn q_n|A_k\geq (1+\epsilon)bn)
&\leq P(\sum_{i=1}^{(1+\epsilon)b n} T_i\leq  (1+\epsilon/2)bn q_n)\\
&\leq P(\sum_{i=1}^{(1+\epsilon)b n} T_i-E(\sum_{i=1}^{(1+\epsilon)b n} T_i)\leq - (\epsilon/4) bn q_n)\\
&\leq P(|\sum_{i=1}^{(1+\epsilon)b n} T_i-E(\sum_{i=1}^{(1+\epsilon)b n} T_i)|\geq (\epsilon/4) bn q_n)\\
&\leq \frac{Var(\sum_{i=1}^{(1+\epsilon)b n} T_i)}{[(\epsilon/4)bn q_n]^2}\\
&=\frac{(1+\epsilon)}{(\epsilon/4)^2 b}\frac{q_n^2-q_n}{nq_n^2}\stackrel{n\to \infty}{\rightarrow} 0.
\end{aligned}\end{equation}

So, the proof of part $(ii)$ is complete. Furthermore, $P(T\geq (1-\epsilon/2)bn q_n|A_k\leq (1-\epsilon)bn)\stackrel{n \to \infty}{\rightarrow} 0$ can be shown in a similar way to ($\ref{eq:t_lemma}$) by using (\ref{eq:aux_lemma_-}), thus also completing the proof of part $(i)$.
\end{proof}

The next lemma guarantees that the original auxiliary process (and so the frog model) lasts for at least some initial rounds whenever $p_n\rightarrow 1$. Later on, this will be important when we show that the particles collected in this early stage are already sufficient to visit a high proportion of vertices.

\begin{lem}\label{lema:r}
Consider the auxiliary process in which $\eta_v$ are i.i.d. variables with the same distribution as $\eta$.
    Suppose $\lim_{n\to \infty}p_n=1$ and $P(\eta=0)<1$. For any $c\in(0,1)$, we have that \[\lim_{n\to \infty}P(R(\mathcal{K}_{n+1},p_n)<k_1)=0,\]
    where $k_1=k_1 (n)=\lfloor cn \rfloor-1$.
\end{lem}

\begin{proof}

The proof is based on the relationship between the auxiliary process in its initial rounds and a branching process. Let $\{Z_j^{(n)}\}_{j \geq 1}$ be a sequence of i.i.d. random variables such that
\begin{equation}\label{eq:dist_zn}
P(Z^{(n)}_j=x)=\begin{cases}
1-p_n \textit{, if $x=0$};\\
p_nc+p_n(1-c)P(\eta=0) \textit{, if $x=1$};\\
p_n(1-c)P(\eta\geq 1) \textit{, if $x=2$}.
\end{cases}
\end{equation}

The distribution of $Z^{(n)}_j$ overestimates $V'_{j-1}(\mathcal{K}_{n+1},p_n)$ when $j\leq k_1$ (compare (\ref{eq:dist_zn}) with (\ref{eq:dist_x}) and (\ref{def:pot_visitados}) realizing that $\frac{p_n(V'_{j-1}(\mathcal{K}_{n+1},p_n)-1)}{n}\leq \frac{p_n(j-1)}{n}\leq \frac{p_n(k_1-1)}{n}\leq p_nc$) and brings together all values greater than 2. So, $\sum_{j=1}^{k}X_j(\mathcal{K}_{n+1},p_n)\succeq\sum_{ j=1}^{k}Z^{(n)}_j$ holds for all $k\in\{1,...,k_1\}$. Recalling the interpretation of $X_j$ as the number of descendants of the particle participating in round $j$, this also implies that
\begin{equation}\label{eq:branch}
    P(R(\mathcal{K}_{n+1},p_n)<k_1)\leq P(R_{Z^{(n)}}<k_1)\leq P(R_{Z^{(n)}}<\infty),
\end{equation} where $R_{Z^{(n)}}$ is the number of rounds until the extinction of a branching process with the number of descendants given by $Z^{(n)}$. For the first inequality in $(\ref{eq:branch})$ to be precise, we should consider that the rounds, as in the auxiliary process, are done by counting descendants of one particle at a time, unlike the usual rounds scheme for branching processes of counting descendants per generation. But as a whole, it does not matter as $P(R_{Z^{(n)}}<\infty)$ is the same for both round schemes. We will prove that $\lim_{n\to \infty}P(R_{Z^{(n)}}<\infty)=0$, thus completing the proof of the lemma as a consequence of $(\ref{eq:branch})$.

Fix any $\epsilon \in(0,1)$. We want to show that $P(R_{Z^{(n)}}<\infty)<\epsilon$ for sufficiently large $n$. By classical properties of a branching process and its extinction probability, we have to show that
\begin{equation}\label{eq:branching}
    s=E(s^{Z^{(n)}})
\end{equation}
has a solution in $(0,\epsilon)$ for sufficiently large $n$.

Equivalently, any root of
\[\begin{aligned}
f_n(s):=E(s^{Z^{(n)}})-s=(1-p_n)+s[p_nc+p_n(1-c)P(\eta=0)-1]+s^2p_n(1-c)P(\eta\geq 1)
\end{aligned}\]
is a solution of $(\ref{eq:branching})$.

We can consider just the cases where $p_n<1$ for sufficiently large $n$, as $p_n =1$ would make $P(Z^{(n)}=0)=0$ and trivially $P(R_{Z^{(n)}}<\infty)=0<\epsilon$. Then, $f_n(0)=1-p_n>0$ and $f_n(s)$ is continuous in $s$. So, to show that $f_n$ has a root in $(0,\epsilon)$, it is enough to show that $f_n(\epsilon)<0$.

As $P(\eta=0)<1$, then $P(\eta=0)+\epsilon P(\eta\geq 1)<1$. Let
\[a:=c+(1-c)[P(\eta=0)+\epsilon P(\eta\geq 1)] \in (0,1).\]

Since $\lim_{n\to \infty}p_n=1$, it is possible to make $(1-p_n)<\epsilon (1-a)/2$ by letting $n$ be large enough, and then
\[f_n(\epsilon)=(1-p_n)+\epsilon (ap_n-1)\leq (1-p_n)-\epsilon (1-a)<0.\]

\end{proof}

\section{Proofs}\label{sec:demo_kn}

Most of the time during this section, the auxiliary process and related variables consider the graph $\mathcal{K}_{n+1}$, survival parameter $p_n$ and that $\{\eta_v\}_{v\in\mathcal{V}_{n+1}}$ are i.i.d. with the same distribution as a variable $\eta$. So, these will be the default to be considered when notation is omitted to make the text cleaner.

We begin by proving the following lemma, which already implies Theorem $\ref{teo:alta_proporcao}$ and which will also help in the proof of the Theorem \ref{teo:transicao_fase}. For this lemma, we use the asymptotic notation $o(1)$ to denote a function $g$ such that $\lim_{n\to \infty}g(n)=0$.

\begin{lem}\label{lema:principal}
Consider that $P(\eta=0)<1$ and $\lim_{n\to \infty} p_n=1$. Fix any $c\in(0,1)$, $\epsilon \in(0,1)$ such that $c<1-\epsilon$. Let $k_1=k_1(n):=\lfloor cn\rfloor-1$ and $k_2=k_2(n):=\inf \{k:V_k(\mathcal{K}_{n+1},p_n)\geq \lceil (1-\epsilon)n\rceil +1\}$ with the definition that $\inf \emptyset=\infty$. Then
\[\lim_{n \to \infty} P(k_2(n)<\infty)=1\]
and for any sequence $(x_n)_{n\in\mathbb{N}}$ of real numbers,
\[P(A_{k_2}(\mathcal{K}_{n+1},p_n)\geq x_n)\geq P(\sum_{i=k_1(n)+1}^{\lceil (1-\epsilon)n\rceil}\eta_i\geq x_n)-o(1).\]
\end{lem}

\begin{obs}
    Lemma $\ref{lema:principal}$ implies Theorem $\ref{teo:alta_proporcao}$ as argued below. Since the only restriction to $(p_n)_{n\in \mathbb{N}}$ in Theorem $\ref{teo:alta_proporcao}$ is that $\lim_{n\to \infty}p_n=1$, we can prove in an equivalent way and without loss of generality that $V_{\infty}(\mathcal{K}_{n+1},p_n)\geq (1-\epsilon)(n+1)$ w.h.p. for any fixed $\epsilon\in(0,1)$. With a fixed $\epsilon\in(0,1)$, we can take $0<\epsilon'<\epsilon$ so that $(1-\epsilon')n\geq (1-\epsilon)(n+1)$ for sufficiently large $n$. By Lemma $\ref{lema:principal}$, we have that w.h.p. $k_2=\inf \{k:V_k(\mathcal{K}_{n+1},p_n)\geq \lceil (1-\epsilon')n\rceil +1\}<\infty$, and therefore w.h.p. $V_\infty(\mathcal{K}_{n+1},p_{n})\geq V_{k_2}(\mathcal{K}_{n+1} ,p_n)\geq (1-\epsilon')n\geq (1-\epsilon)(n+1)$. 
\end{obs}

\begin{proof}

In order to avoid some technical problems in the proof of Lemma $\ref{lema:principal}$, we consider only cases in which $p_n$ is non-decreasing and $p_n\neq 1$ for all $n\in\mathbb{N}$ in addition to the original assumption of $\lim_{n\to\infty} p_n=1$. These additional assumptions do not reduce the generality of $(p_n)_{n\in\mathbb{N}}$. In fact, if $(p_n)_{n\in\mathbb{N}}$ is any sequence such that $\lim_{n\to\infty} p_n=1$, we can construct another sequence $(p'_n )_{n\in\mathbb{N}}$ such that $p'_n:=\min\{1-\frac{1}{n},\inf_{k\geq n}\{p_k\}\}$. Since $\lim_{n\to\infty}p'_n=1$, $p'_n$ in non-decreasing and $p'_n\neq 1$ for all $n\in\mathbb{N}$, the lemma would be valid for $p'_n$ and consequently also for $p_n$, because $p_n\geq p'_n$ implies $V_k(\mathcal{K}_{n+1},p_n)\succeq V_k(\mathcal{K}_{n+1},p'_n)$ and $ A_k(\mathcal{K}_{n+1},p_n)\succeq A_k(\mathcal{K}_{n+1},p'_n)$ for all $k\in\mathbb{N}$.

Initially, suppose that $E(\eta)<\infty$. Let $\{Y_j^{(n)}\}_{j \in \mathbb{N}}$ be a sequence of i.i.d. random variables such that
\begin{equation}\label{eq:dist_yn}
P(Y^{(n)}_j=x)=\begin{cases}
1-p_n \textit{, if $x=0$}\\
p_nc+p_n(1-c)P(\eta=0) \textit{, if $x=1$}\\
p_n(1-c)P(\eta=1) \textit{, if $x=2$}\\
\vdots\\
p_n(1-c)P(\eta=k-1) \textit{, if $x=k$}\\\\
\vdots\\
\end{cases}
\end{equation}

Note that $E(Y^{(n)}_j)=p_nc+p_n(1-c)[E(\eta)+1]\stackrel{n \to \infty}{\rightarrow} c+(1-c )[E(\eta)+1]$ as $\lim_{n\to \infty} p_n=1$. Since $P(\eta=0)<1$ and therefore $E(\eta)>0$, we have that $c+(1-c)[E(\eta)+1]=1+(1-c)E(\eta)>1$. So, if we let $d:=(1-c)E(\eta)/2>0$, then
\begin{equation}\label{ineq:ey}
E(Y_j^{(n)})\geq 1+d\text{ for sufficiently large } n.\end{equation}

Let $b>0$ be a constant small enough so that $\frac{(1+\epsilon)b}{c}\leq d/4$ and therefore $\frac{(1+\epsilon)bn}{k_1}\leq \frac{(1+\epsilon)bn}{cn-2}\leq d/2$ for sufficiently large $n$. Using this fact and (\ref{ineq:ey}), we conclude that there is a $n_0$ such that the following inequality, which will be used later, holds:
\begin{equation}\label{ineq:n0}
    \frac{(1+\epsilon)bn}{k_1}+1-E(Y_j^{({n_0})})\leq -d/2<0,\text{ for }n\geq n_0.
\end{equation}

The distribution of $Y^{(n)}_j$ overestimates $V'_{j-1}(\mathcal{K}_{n+1},p_n)$ when $j\leq k_1=\lfloor cn \rfloor-1$ (compare (\ref{eq:dist_yn}) with (\ref{eq:dist_x}) and (\ref{def:pot_visitados}) realizing that $\frac{p_n(V'_{j-1}(\mathcal{K}_{n+1},p_n)-1)}{n}\leq \frac{p_n(j-1)}{n}\leq \frac{p_n(k_1-1)}{n}\leq p_nc$), so the comparison $\sum_{j=1}^{k_1}X_j(\mathcal{K}_{n+1},p_n) \succeq \sum_{ j=1}^{k_1}Y^{(n)}_j$ holds. Recall that we are considering that $p_n$ is non-decreasing, so
\begin{equation}\label{ineq:estoc_n0}
\sum_{j=1}^{k_1}X_j(\mathcal{K}_{n+1},p_n) \succeq\sum_{j=1}^{k_1}Y^{(n)}_j \succeq \sum_{ j=1}^{k_1}Y^{({n_0})}_j\text{ for }n\geq n_0.\end{equation}

When $n\geq n_0$, we can use the definition in $(\ref{def:pot_ativas})$ as well as (\ref{ineq:n0}) and (\ref{ineq:estoc_n0}) to conclude that
\[\begin{aligned}
P(A'_{k_1}(\mathcal{K}_{n+1},p_n)\geq (1+\epsilon)bn)&= P(\sum_{j=1}^{k_1} X_j(\mathcal{K}_{n+1},p_n) \geq (1+\epsilon)bn+k_1-1-\eta_0)\\
&\geq P(\sum_{j=1}^{k_1} Y^{(n_0)}_j \geq (1+\epsilon)bn+k_1)\\
&\geq P(\frac{\sum_{j=1}^{k_1} Y^{(n_0)}_j-E(\sum_{j=1}^{k_1} Y^{(n_0)}_j)}{k_1} \geq -d/2),
\end{aligned}\]
which implies that $A'_{k_1}\geq (1+\epsilon)bn$ w.h.p. by the law of large numbers as $k_1(n)=\lfloor cn\rfloor-1\to +\infty$. 

By Lemma \ref{lema:r}, we can guarantee that $A_{k_1}=A'_{k_1}$ w.h.p. (compare with $\ref{def:ativas}$), meaning that the original process does not die before round $k_1$ and there are indeed at least $(1 +\epsilon)bn$ active particles at the end of this round, not just potentially active particles.

Let us now recall we concluded that $A_{k_1}\geq (1+\epsilon)bn$ w.h.p. only for the case in which $E(\eta)<\infty$, in addition to the original assumption of Lemma $\ref{lema:principal}$ that $P(\eta=0)<1$. The case $E(\eta)=\infty$ can be obtained by comparison, using $\eta^{*}:=\mathbbm{1}(\eta\geq 1)$. We temporarily use the notations of $A_{k_1}$ and $A_{k_1}^*$ to differentiate the models using $\eta$ and $\eta^*$, respectively. As $P(\eta^*=0)<1$ and $E(\eta^{*})<\infty$, then w.h.p. $A_{k_1}^*\geq (1+\epsilon)bn$. Combining this with the fact that $A_{k_1}\succeq A_{k_1}^*$, we conclude that $A_{k_1}\geq (1+\epsilon)bn$ w.h.p. also in the case $E(\eta )=\infty$.

Let us also recall that the auxiliary process on $\mathcal{K}_n$ and all equations shown related to it do not depend on the rule for choosing which active particle is selected to participate in each round. So, now we make a non-mandatory modification of such rule whose sole purpose is to make the proof possibly easier to understand. For any $k\geq k_1$, it always happens that one of the $A_{k_1}$ active particles alive at time $k_1$ is chosen to participate in round $k$, unless there is no longer any alive active particle at time $k$ that was also active at time $k_1$. In other words, particles activated after $k_1$ have to wait until all particles that were active at time $k_1$ die before having the chance to participate in a round.

Let $T(k_1)$ be the total combined number of remaining steps that all $A_{k_1}$ active particles alive at time $k_1$ have yet to take. Note that w.h.p. $T(k_{1})\geq (1+\epsilon/2)bn (1-p_n)^{-1}$ by Lemma \ref{lema:vida_total} part $(ii)$. Also, $\tau_{\lceil(1-\epsilon)n\rceil}\leq n[2(\frac{1-\epsilon } {\epsilon})]$ holds w.h.p. in the coupon collector's problem by Lemma $\ref{lema:ccp_parcial}$. As $(1-p_n)^{-1}$ diverges,  we have that w.h.p.\[T(k_1)\geq n (1-p_n)^{-1}(1+\epsilon/2)b\geq n[2(\frac{1-\epsilon }{\epsilon})]\geq \tau_{\lceil(1-\epsilon)n\rceil}.\]
Therefore, by the comparison between the coupon collector's problem with $n$ coupons and the auxiliary process on $\mathcal{K}_{n+1}$, the active particles alive at time $k_1$ w.h.p. jointly visit at least $\lceil (1-\epsilon)n\rceil +1$ vertices. This completes the first part of the proof, as w.h.p. $k_2=\inf \{k:V_k=\lceil (1-\epsilon)n\rceil +1\}<\infty$ (or in other words, $\{k:V_k=\lceil (1-\epsilon)n\rceil +1\}\neq \emptyset$).

Note that by our imposed rule for selecting which particle participates in each round, w.h.p. only particles that were activated up to instant $k_1$ participate in rounds $k\in(k_1,k_2]\cap \mathbb{N}$. Therefore, the event
\[E_n:=\{\text{every particle activated in rounds $k\in(k_1,k_2]\cap \mathbb{N}$ is alive at the end of the round $k_2$}\}
\] 
holds w.h.p., since it is not possible for a particle to die without having participated in some round.

Fix any sequence of real numbers $(x_n)_{n\in\mathbb{N}}$. Note that the number of particles activated in rounds $k\in(k_1,k_2]\cap \mathbb{N}$ has the same distribution as $\sum_{i=V_{k_1}+1}^{\lceil (1-\epsilon)n\rceil+1}\eta_i$ and $\sum_{i=V_{k_1}}^{\lceil (1-\epsilon)n\rceil}\eta_i$. So, $A_{k_2}$ conditioned on $E_n$ is stochastically greater than $\sum_{i=V_{k_1}}^{\lceil (1-\epsilon)n\rceil}\eta_i$. Furthermore, by the definitions in (\ref{def:pot_visitados}) and ($\ref{def:visitados}$), we have that $V_{k_1}\leq k_1+1$. Therefore, for all $n$,
\[
P(A_{k_2}\geq x_n|E_n)\geq P(\sum_{i=V_{k_1}}^{\lceil (1-\epsilon)n\rceil}\eta_i\geq x_n)\geq P(\sum_{i=k_1+1}^{\lceil (1-\epsilon)n\rceil}\eta_i\geq x_n).\]

So, we conclude that
\[\begin{aligned}
P(A_{k_2}\geq x_n)&\geq P(A_{k_2}\geq x_n|E_n)P(E_n)\\
&\geq P(\sum_{i=k_1+1}^{\lceil (1-\epsilon)n\rceil}\eta_i\geq x_n)[1-P(E^c_n)]\\
&= P(\sum_{i=k_1+1}^{\lceil (1-\epsilon)n\rceil}\eta_i\geq x_n)-P(\sum_{i=k_1+1}^{\lceil (1-\epsilon)n\rceil}\eta_i\geq x_n)P(E_n^c),
\end{aligned}\]
completing the proof of Lemma \ref{lema:principal}, as $P(\sum_{i=k_1+1}^{\lceil (1-\epsilon)n\rceil}\eta_i\geq x_n)P(E_n^c)\leq P(E_n^c)=o(1)$.

\end{proof}

\begin{proof}[Proof of Theorem \ref{teo:transicao_fase} part $(i)$]

The proof strategy used here is similar to that used in \cite[Proposition 1.1]{frogs_on_trees}. 

To be able to use the lemmas directly, we first show that $\lim_{n \to \infty}P( V_\infty(\mathcal{K}_{n+1},p_n)=n+1)=0$ and then we show that the result also holds when exchanging $p_n$ for $p_{n+1}$. So, we use the graph $\mathcal{K}_{n+1}$ and we let $\mathcal{V}=\mathcal{V}_{n+1}$ be its corresponding set of vertices, with $|\mathcal{V}|=n+1$.

The idea is to place all $1+\sum_{v\in \mathcal{V}}\eta_v$ original particles at the root $o\in\mathcal{V}$, so that every particle starts out active. In other words, we deal with another auxiliary process, but considering $1+\eta^*_o=1+\sum_{v\in \mathcal{V}}\eta_v$ and $\eta^*_v=0$ for all $ v\in\mathcal{V}\setminus\{o\}$. From now on, we use the asterisk to differentiate the auxiliary process that uses $\{\eta^*_v\}_{v\in\mathcal{V}}$ from the one that uses $\{\eta_v\}_{v \in\mathcal{V}}$.

We have that $P(V^*_\infty(\mathcal{K}_{n+1},p_n)=n+1)\geq P(V_\infty(\mathcal{K}_{n+1},p_n) =n+1)$ by the structure of the complete graph, as argued below. Note that, conditioned on $V'_{j-1}=V'^*_{j-1}=v$, in round $j$ we have that the probability of visiting a new vertex is $p\frac{ v-1}{n}$ in both processes; the probability of visiting a previously visited vertex is $p\frac{n-v+1}{n}$ in both processes; the probability of the particle dying is $1-p$ in both processes. Therefore, we can couple these two models so that $V'_k=V'^*_k$ for all $k\in\mathbb{N}$. Additionally, note that $A'^*_k=1+\sum_{v\in\mathcal{V}}\eta_v-\sum_{i=1}^k \mathbbm{1}_{(X^*_i =0)}$ since all particles start out active and a change in the number of active particles occurs only in the case of death (compare with (\ref{def:pot_ativas})); therefore $A'^*_k\geq A'_k$ for all $k\in\mathbb{N}$. Then, $R^*=\inf\{k:A'^*_k=0\}\geq \inf\{k:A'_k =0\}= R$, that is, agglutinating the particles at the root at the initial instant causes the model to survive for a greater or equal number of rounds. Bringing all this information together, we conclude that $V^*_\infty=V'^*_{R^*}\geq V'^*_{R}= V'_R= V_\infty$.

Our goal is to show that, even when we accumulate all the initial particles at the root, it still holds that $\lim_{n\to \infty}P(V^*_\infty(\mathcal{K}_{n+1},p_n)=n+1)=0$.

Let $A^*_0=1+\sum_{v\in\mathcal{V}} \eta_v$ be initial number of active particles at the root of $\mathcal{K}_{n+1}$. Fix $\epsilon>0$ such that $(1-\epsilon)\alpha-E(\eta)>\epsilon$. Since $|\mathcal{V}|=n+1$, we have that
\[\begin{aligned}
P(A^*_0\leq (1-\epsilon)\alpha n)&=P(\frac{1+\sum_{v\in \mathcal{V}} \eta_v}{|\mathcal{V}|}\leq \frac{(1-\epsilon)\alpha n}{n+1})\\
&=P(\frac{\sum_{v\in \mathcal{V}}\eta_v}{|\mathcal{V}|}-E(\eta)\leq \frac{(1-\epsilon)\alpha n}{n+1}-E(\eta)-\frac{1}{n+1})\\
\end{aligned}\]
has $1$ as a limit when $n\to \infty$ by the law of large numbers, as $\frac{(1-\epsilon)\alpha n}{n+1}-E(\eta)- \frac{1}{n+1}>\epsilon/2$ holds for sufficiently large $n$.

Thus, by Lemma \ref{lema:vida_total} part $(i)$, we know that all $A^*_0$ particles together make w.h.p. at most $(1-\epsilon/2)n\log n$ steps. Now, we use the equivalence with the coupon collector's problem and Lemma \ref{lema:ccp_concentrado} to conclude that this number of steps w.h.p. is not enough to visit all the vertices of $\mathcal{K}_{n+1}$.

We proved that $\lim_{n \to \infty}P( V_\infty(\mathcal{K}_{n+1},p_n)=n+1)=0$ for all $\alpha>E(\eta)$. Below we show how the same would apply when using $p_{n+1}$, finishing the proof of Theorem \ref{teo:transicao_fase} part $(i)$.

Write $p_{n}(\alpha)=1-\frac{\alpha}{\log(n)}$. For any fixed $\alpha>E(\eta)$, select some $\alpha^*\in ( E(\eta),\alpha)$. As $\lim_{n\to \infty}\frac{\log (n+1)}{\log (n)}=1$, for a sufficiently large $n$, we have that
\[\frac{\alpha}{\alpha^*}\geq \frac{\log(n+1)}{\log(n)},
\]
which implies that $p_{n+1}(\alpha)=1-\frac{\alpha}{\log(n+1)} \leq 1-\frac{\alpha^*}{\log(n)} =p_{n}(\alpha^*)$ and, due to the monotonicity of the model in $p$, also that $P(V_\infty(\mathcal{K}_{n+1},p_{n+1}(\alpha))= n+1)\leq P(V_\infty(\mathcal{K}_{n+1},p_{n}(\alpha^*))=n+1)$.

As we have already shown that part $(i)$ of Theorem \ref{teo:transicao_fase} holds for $p_{n}(\alpha^*)$, it also holds for $p_{n+1}(\alpha)$.

\end{proof}

\begin{proof}[Proof of Theorem \ref{teo:transicao_fase} part $(ii)$]

We only focus on the case where $E(\eta)<\infty$. It is easy to see that proving Theorem \ref{teo:transicao_fase} part $(ii)$ in the case $E(\eta)<\infty$ already implies proving it also for the case $E(\eta)=\infty$. To check this, we start with $\eta$ such that $E(\eta)=\infty$ and compare it with $\eta^{(m)}:=\eta\mathbbm{1}_{\{\eta\leq m\}} +m\mathbbm{1}_{\{\eta> m\}}$. Since $\lim_{m\to \infty}E(\eta^{(m)})=E(\eta)=\infty$ by the Monotone Convergence Theorem, for any $\alpha>0$, we can take a sufficiently large $m$ such that $\alpha<E(\eta^{(m)})$. The comparison $\eta\succeq \eta^{(m)}$ and this same Theorem supposedly valid for $\eta^{(m)}$ (since $\alpha<E(\eta^{(m)})<\infty$) imply that $P_\eta(V_\infty(\mathcal{K}_n,p_n)=n)\geq P_{\eta^{(m)}}(V_\infty (\mathcal{K}_n,p_n)=n)\stackrel{n\to \infty}{\to}1$.

Furthermore, to be able to use the lemmas directly, we focus only on the proof that $\lim_{n\to \infty} P (V_\infty(\mathcal{K}_{n+1}, p_n) = n +1) = 1$. The monotonicity of the frog model in $p$ implies that this is enough to demonstrate the same for $p_{n+1}\geq p_n$.

As $\alpha<E(\eta)$, we can fix $\epsilon\in(0,1)$ and $c\in(0,1)$ such that $c<1-\epsilon$ and $\frac{(1+\epsilon)\alpha}{1-\epsilon-c}<E(\eta)$. 

Since $(1-\epsilon-c)n\leq \lceil (1-\epsilon)n\rceil-\lfloor cn \rfloor+1\leq (1-\epsilon-c)n+3$, we have that
\begin{equation}\label{ineq:comp_ep_alp}
\frac{(1+\epsilon)\alpha n-E(\sum_{i=\lfloor cn\rfloor }^{\lceil (1-\epsilon)n\rceil}\eta_i)}{\lceil (1-\epsilon)n\rceil-\lfloor cn \rfloor+1}\leq \frac{(1+\epsilon)\alpha}{1-\epsilon-c}-\frac{(1-\epsilon-c)n E(\eta)}{(1-\epsilon-c)n+3}.
\end{equation}

By Lemma \ref{lema:principal} and (\ref{ineq:comp_ep_alp}), if we let $k_2:=\inf \{k:V_k(\mathcal{K}_{n+1}, p_n)=\lceil (1-\epsilon)n \rceil+1\}$, then
\[\begin{aligned}
P(A_{k_2}\geq (1+\epsilon)\alpha n)&\geq P(\sum_{i=\lfloor cn\rfloor}^{\lceil(1-\epsilon)n\rceil}\eta_i\geq (1+\epsilon)\alpha n)-o(1)\\
&\geq P(\frac{\sum_{i=\lfloor cn\rfloor}^{\lceil(1-\epsilon)n\rceil}\eta_i-E(\sum_{i=\lfloor cn\rfloor}^{\lceil(1-\epsilon)n\rceil}\eta_i)}{\lceil (1-\epsilon)n\rceil-\lfloor cn \rfloor+1}\geq \frac{(1+\epsilon)\alpha}{1-\epsilon-c}-\frac{(1-\epsilon-c)n E(\eta)}{(1-\epsilon-c)n+3})-o(1).
\end{aligned}\]

For the initially chosen constants $\epsilon$ and $c$, we have that $z:=\frac{(1+\epsilon)\alpha}{1-\epsilon-c} -E(\eta)<0$. It is also true that $\frac{(1+\epsilon)\alpha}{1-\epsilon-c}-\frac{(1-\epsilon-c)n E(\eta)}{(1-\epsilon- c)n+3}\stackrel{n\to \infty}{\rightarrow} z$ and therefore $\frac{(1+\epsilon)\alpha}{1-\epsilon-c}-\frac{(1 -\epsilon-c)n E(\eta)}{(1-\epsilon-c)n+3}\leq z/2<0$ for sufficiently large $n$. Then, we can apply the law of large numbers to conclude that w.h.p. $A_{k_2}(\mathcal{K}_{n+1},p_n)\geq (1+\epsilon)\alpha n$.

By Lemma \ref{lema:vida_total} part $(ii)$, we know that w.h.p. these $A_{k_2}$ particles jointly take at least $(1+\epsilon/2)n\log n$ steps. By Lemma \ref{lema:ccp_concentrado} and the comparison with the coupon collector's problem, this number of steps is w.h.p. enough to visit all vertices of $\mathcal{K}_{n+1}$.
\end{proof}

\bibliographystyle{alpha}
\bibliography{Cov-Jul2024}

\end{document}